\documentclass{amsart}

\usepackage{hyperref}
\usepackage{mathdots}

\usepackage{graphicx} 
\usepackage{amssymb}
\usepackage[all]{xy}
\usepackage{mathrsfs} 
\usepackage{mathtools}
\usepackage{algorithm2e}
\RestyleAlgo{ruled}
\usepackage{tikz,tkz-euclide}
\usepackage{MnSymbol}

\usepackage{array}
\usepackage[backrefs,lite]{amsrefs}

\hypersetup{
    colorlinks,    citecolor=blue,    filecolor=blue,    linkcolor=blue,    urlcolor=blue
}

\usepackage{tikz}
\usepackage{tikz-cd}
\usepackage{longtable}

\hypersetup{
    colorlinks,    citecolor=blue,    filecolor=blue,    linkcolor=blue,    urlcolor=blue
}

\usepackage{booktabs}
\usepackage{mathrsfs}
\usepackage{todonotes}
\usepackage{bbm}

\def\rr{{\mathbb R}}

\def\zz{{\mathbb Z}}
\def\qq{{\mathbb Q}}
\def\pp{{\mathbb P}}

\newcommand{\Eff}{\operatorname{Eff}}
\newcommand{\CF}{\operatorname{CF}}
\newcommand{\Nef}{\operatorname{Nef}}
\newcommand{\N}{\operatorname{N}}

\newcommand{\Bl}{\operatorname{Bl}}

\usetikzlibrary{arrows}

\newtheorem{introthm}{Theorem}
\newtheorem{intropro}{Proposition}

\newtheorem{introque}[introthm]{Question}

\newtheorem{theo}{Theorem}[section]
\newtheorem{lem}[theo]{Lemma}
\newtheorem{prop}[theo]{Proposition}
\newtheorem{cor}[theo]{Corollary}
\newtheorem{defin}[theo]{Definition}

\theoremstyle{definition}

\newtheorem{quest}[theo]{Question}
\newtheorem{rem}[theo]{Remark}

\newcommand{\bEff}{\overline{\Eff}}

\title{ \bf Cones of cycles on
blowups of $(\pp^1)^n$}
\author{Gilberto Bini, Luca Ugaglia}

\address{
Dipartimento di Matematica e Informatica,
Universit\`a degli studi di Palermo,
Via Archirafi 34,
90123 Palermo, Italy}
\email{luca.ugaglia@unipa.it,\, gilberto.bini@unipa.it}

\thanks{The authors have been partially supported by 
``Piano straordinario per il miglioramento della qualit\`a della ricerca e dei risultati della VQR 2020-2024 - Misura A'' and research funds ``FFR'' of the University of Palermo.
Both authors are members of 
INdAM - GNSAGA}

\begin{document}

\begin{abstract}
We study cones of pseudoeffective cycles on
the blow up of $(\pp^1)^n$ at points in very general
position, proving some results concerning their structure. In particular we show that in some cases they turn out to be generated by exceptional classes and 
fiber classes relatively to the projections onto a smaller number of copies of projective lines. 
\end{abstract}

\maketitle



\section*{Introduction}

Cones of curves and divisors of an
algebraic variety $X$ are closely related to 
the birational geometry of $X$ and they
have played a central role since the 
work of Mori in 1980's. 

More recently, the interest in 
cycles of intermediate dimensions
has increased, and 
there has been significant progress in understanding the cones of such cycles (see for 
instance~\cites{DELV,FL,O}).
Nevertheless, the number of examples in which cones of effective cycles have been explicitly computed is relatively small. The most significant results concern blowups of projective space $\pp^n$
at sets of points~\cite{CLO}.
In this case, if the number $r$ of points
is small, the blow up $X$ is a toric variety, 
so that 
the pseudoeffective cone $\bEff_k(X)$ of $k$-dimensional cycles is polyhedral for any 
$1\leq k\leq n-1$, and generated by classes
of torus-invariant subvarieties 
(see for instance~\cite{P}). 
These are linear classes, i.e. classes of 
strict transforms of $k$-dimensional linear spaces 
containing some of the points, and
$k$-dimensional linear spaces contained in the 
exceptional divisors. Starting from this observation, the authors show that the
cone $\bEff_k(X)$ is still generated
by these linear classes even if the number 
of points is increased
beyond the toric range, and in particular 
they give bounds on $k,\,r$ and $n$ that
guarantee this property.

In this paper we consider the analogous problem 
for the blow up $X_r^n$ of $(\pp^1)^n$ 
at a set of $r$ very general points. 
Blowing up a small number of points (up to $2$) 
in $(\pp^1)^n$ we still have a toric variety, 
so that the cone $\bEff_k(X_r^n)$ is generated by 
classes of torus-invariant subvarieties. 
It turns out that these are the linear spaces in the exceptional divisors and the  
strict transforms of the fibers of the projections
$(\pp^1)^n\to (\pp^1)^{n-k}$ (see 
Proposition~\ref{prop:tor}). We then say that 
the cone $\bEff_k(X_r^n)$ is {\em fiber-generated} (Definition~\ref{def:fib}).

Therefore, a natural problem is to investigate the values of 
$n,\,k$ and $r$ for which $\bEff_k(X_r^n)$ is 
fiber-generated. If so, the latter is in particular 
rational polyhedral, as well as the dual cone ${\rm Nef}^k(X_r^n)$ of codimension $k$ nef cycles, thus giving more classes of varieties as in Problem 6.9 
in~\cite{DELV}.

Our first result concerns the
``extremal'' cases, i.e. $k = n-1$ or $k = 1$, when 
it is possible to give a complete characterization.
\begin{introthm}
 Let $n$ be any integer greater than one.
 \begin{enumerate}
     \item (Corollary \ref{cor:div}) $\bEff_{n-1}(X_n^r)$ is fiber-generated
     if and only if $r\leq 2$.
     \item (Theorem \ref{thm:linear}) $\bEff_1(X_n^r)$ is fiber-generated if and
     only if $r\leq n!$.
 \end{enumerate}
\end{introthm}
The analogous result for the blow-up 
$Y_r^n$ of
$\mathbb P^n$ at $r$ general points is the
following:
$\bEff_{n-1}(Y_r^n)$ is linearly generated 
if and only if $r\leq n+2$ (\cite{CLO}*{Thm.~2.7});
$\bEff_{1}(Y_r^n)$ is linearly generated 
if and only if $r\leq 2^n$ (\cite{CLO}*{Prop.~4.1}).

The theorem above suggests that the number $r$ of points
that can be blown up in order to have fiber-generated
cones decreases as the dimension $k$
of the cycles increases. This is confirmed by the
following result, which gives a bound for any 
$1\leq k\leq n-1$ (the characterization given for $k=1$
shows that in general this bound is not sharp).
\begin{introthm}
\label{thm:bound}
(Theorem \ref{thm:lb})
 For any $1\leq k\leq n-1$, the pseudoeffective
 cone $\bEff_k(X_n^r)$ is fiber-generated if 
 $r \leq n-k+1$.
\end{introthm}
On the negative side, we prove that for any $n$ and $k$, 
if we take a sufficiently large number $r$ of points,
the cone $\bEff_k(X_r^n)$ is not fiber-generated.
\begin{intropro}
(Proposition~\ref{prop:not})
The cone $\bEff_k(X_r^n)$ is not fiber-generated for $r > (n-k+1)!$.
\end{intropro}
For $k = 1$ or $n-1$ the above bound 
is sharp, while for the other values
of $k$ we are not able to check the
sharpness. We remark that when $n=4$, combining Theorem~\ref{thm:bound} 
with the above proposition,
only three cases remain open, i.e. 
$\bEff_2(X_r^4)$, for $r = 4,5,6$. 
In Proposition~\ref{prop:4-4} we will deal with
the case $r = 4$, showing that
$\bEff_2(X_4^4)$ is still fiber-generated. 

Finally we recall that if $X$ is a {\em Mori dream 
space} (see~\cite{HK}), then the cone $\bEff_{n-1}(X)$ is
polyhedral, but in general the cones $\bEff_k(X)$ 
are not necessarily polyhedral for $1\leq k\leq n-2$ (see for 
instance~\cite{DELV}*{Ex.~6.2}).
Nevertheless, in~\cite{CLO} it is proved that if
$X$ is the blow up of $\pp^n$ at $r$ general points, 
in all the cases in which $X$ is Mori dream,
$\bEff_k(X)$ turns out to be polyhedral for any $k$.
Therefore it makes sense to consider the analogous
question in our case.
\begin{introque}
\label{que:md}
If $X_r^n$ is a Mori dream space, 
is the cone $\bEff_k(X_r^n)$ polyhedral
for any $1\leq k\leq n-1$?
\end{introque}

We close our paper with some comments about the 
above question, showing that our results 
give a partial answer.

The authors would like to thank the referee for 
many useful remarks that helped improving the
quality of the paper.

\section{Preliminaries}
Let $X$ be a smooth projective variety of dimension $n$. 
Let $Z_k(X)$ denote the group of algebraic cycles of dimension $k$ on $X$. 
We define the space of numerical classes of $k$-cycles
\[
 N_k(X) := (Z_k(X)/\equiv)\otimes \mathbb R
\]
where $\equiv$ is the numerical equivalence. 
For any $1\leq k\leq n-1$, $N_k(X)$ 
is a finite-dimensional
vector space, and intersection gives a perfect pairing
\[
N_k(X) \times N_{n-k}(X) \to \mathbb R.
\]
Given a $k$-dimensional subvariety $Z\subseteq X$, we write $[Z]$ to denote its class in $N_k(X)$. 
We recall that if $Y$ and $Z$ are subvarieties of $X$
of dimension $k$ and $n-k$ respectively, and 
$Y\cap Z$ is a finite set, then $[Y]\cdot [Z] \geq 0$.

A class $\omega\in Z_k(X)$ is called {\em effective}
if there exist subvarieties $Z_1,\dots,Z_s\subseteq X$
and real positive numbers $\alpha_1,\dots,\alpha_s$ such that
$\omega =\sum_i\alpha_i[Z_i]$. 
We denote by $\bEff_k(X)\subseteq N_k(X)$ 
the {\em pseudoeffective cone}
of $k$-dimensional classes, i.e. the cone 
containing the (limits of) classes of 
$k$-dimensional effective cycles.

A class 
$\omega\in Z_k(X)$ is called {\em nef} if $\omega\cdot
[Y]\geq 0$, for any subvariety $Y\subseteq X$ of 
codimension $k$, or equivalenty $\omega\cdot\alpha \geq 0$,
for any effective class $\alpha\in Z_{n-k}(X)$.
We denote by $\Nef_k(X)$ the {\em nef cone}, i.e. the
cone of nef $k$-dimensional classes. If we write $\Nef^k(X)$ we refer to codimension $k$ nef classes.

\vspace{2mm}

 Let now $\Gamma$ be a set of $r$ distinct
 points $q_1,\dots,q_r$ in $(\pp^1)^n$. In what 
 follows, we will denote by 
 $\pi \colon X^n_{\Gamma}\to (\pp^1)^n$ the 
 blow up of $(\pp^1)^n$ at 
 the points in $\Gamma$, with
 exceptional divisors 
 $E_{1},\dots,E_{r}$. We recall that the points
 $q_i$ are in {\em very general position}
 (or equivalently they are {\em very general})
 if their configuration
 lies in the complement of a countable union of
 proper configurations. If this is the case, we 
 will simply write $X^n_r$ instead of $X_{\Gamma}^n$. 
 Given a subset $I \subseteq \{1,\dots,n\}$ of 
 cardinality $|I| = n-k$ we
 denote by $p_I:\, (\pp^1)^n \to (\pp^1)^{n-k}$
 the projection defined by
 \[
 ((x_{1}:y_{1}),\dots,(x_{n}:y_{n}))
 \mapsto
 ((x_{i}:y_{i}) \mid\, i\in I).
 \]
Moreover we set $H_I\in \N_k(X^n_{\Gamma})$ to be 
the class of the
pull back of a fiber of $p_I$ via the blow-up $\pi$
and $E_{j,k}\in \N_k(X^n_{\Gamma})$ to be the
class of a linear space of dimension $k$ 
in the exceptional divisor $E_j$.
Therefore, for any $1\leq k\leq n-1$, the vector space $\N_k(X^n_{\Gamma})$ has dimension $\binom{n}{k}+r$,
with basis
\[
\{H_I\,|\, |I| = n-k\}\cup \{E_{j,k}\,|\, 
1\leq j\leq r\}.
\]
When $|I| = 1$ we will simply write
$H_i$ instead of $H_{\{i\}}$.

Let $Y$ be a $k$-dimensional irreducible subvariety in
$X^n_{\Gamma}$. Reasoning as in~\cite{CLO}*{Lem.~2.3}
we have that the following holds:
\begin{itemize}
    \item[-] if $Y\subseteq E_j$ for some $1\leq j\leq r$,
    then $[Y] = b_jE_{j,k}$ for some $b_j > 0$;
    \item[-] if $Y$ is not contained in any exceptional 
    divisor then the numerical class of $Y$ is given by $[Y] = \sum_{|I| = n-k}a_IH_I -
    \sum_{j=1}^rb_jE_{j,k}$, with $a_I,b_j \geq 0$.
\end{itemize}
\begin{rem}
\label{rem:int}
From the intersection formulas on $(\pp^1)^n$ and from
the fact that every $E_j$ is isomorphic to $\pp^{n-1}$, 
we have that for any $I,J\subseteq\{1,\dots,n\}$ and 
$1\leq j\leq r$,
\[
H_I\cdot H_J = 
\left\{
\begin{array}{ll}
H_{I\cup J}, & \text{if } I\cap J = \emptyset\\
0, & \text{otherwise,}
\end{array}
\right.
\quad
(-1)^{n-k+1}E_{j}^{n-k} = E_{j,k},
\quad
E_j\cdot E_{j,k} = -E_{j,k-1}.
\]
In particular, when $|I| = n-k$
and $|J| = k$ we deduce that 
\[
H_I\cdot H_J = 
\left\{
\begin{array}{ll}
1, & \text{if } I\cup J = \{1,\dots,n\}\\
0, & \text{otherwise,}
\end{array}
\right.
\quad
H_I\cdot E_{j,n-k} = 0,
\qquad
E_{j,n-k}\cdot E_{j,k} = -1
\]
Therefore, if we choose
a suitable order on the generators, the 
matrix of the pairing 
$N_k(X_{\Gamma}^n)\times 
N_{n-k}(X_{\Gamma}^n) \to \rr$ 
is the diagonal join of $I_N$ and $-I_r$, where 
$N = \binom{n}{k}$.
\end{rem}
We recall that for any $n\geq 2$, given $n$ 
points $p_1,\dots,p_n\in\pp^n$ in very 
general position there 
exists an isomorphism in codimension $1$
 \begin{equation}
 \label{eq:1}
 \varphi\,
 \colon 
 \Bl_{p_1\dots p_n}
 \pp^n
  \dashedrightarrow
 \Bl_{q_1}(\pp^1)^n.
 \end{equation}
 Therefore, for any $r\geq 1$ we have an
 isomorphism in codimension $1$
 between the blow up of $\pp^n$
 at $n+r-1$ very general points and 
 the blow up $X_r^n$ of $(\pp^1)^n$
 at $r$ very general points. If we denote by 
 $\mathcal H,\mathcal E_1,\dots,\mathcal E_{n+r-1}$,
 a basis for the Picard group of the blow up of
 $\pp^n$, the isomorphism on the Picard groups 
 induced by $\varphi$ can
 be described as follows (see~\cite{LM})
 \[
 \left\{
 \begin{array}{llll}
  \mathcal H & \mapsto & \sum_{i=1}^n
  H_i - (n-1)E_1 &  \\
  \mathcal E_i & \mapsto & H_{n+1-i} - 
  E_1 & \text{ for } 1\leq i\leq n\\
  \mathcal E_i & \mapsto & E_{i-n+1} & \text{ for } i \geq n+1.\\
  \end{array}
 \right.
 \]
 In particular, the pseudoeffective cone
 $\bEff_{n-1}(X^n_r)$ is isomorphic
 to the pseudoeffective cone of
 the blow up of $\pp^n$ at $n+r-1$ very general points.

 Since $\varphi$ and its birational inverse contract
 varieties of codimension $\geq 2$, they can not be
 used in order to study the cones 
 $\bEff_k(X^n_r)$ when $k\leq n-2$.

\vspace{5mm}

For the purposes of what follows, we briefly recall the definition of multiplicity of a subvariety in an ambient variety, as well as that of intersection multiplicity along a subvariety.

Let $M$ be a smooth ambient variety of dimension $n$. Assume $Z$ is a subvariety of $M$. Then {\em the multiplicity $e_Z(M)$ of $M$ along $Z$} is the coefficient of the class $[Z]$ in the (total) Segre class $s(Z,M)$, i.e. the Segre class of the normal cone $C_Z M$ to $Z$ in $M$. As explained in \cite{Ful}*{Ex.~4.3.1}, this definition is equivalent to that of the multiplicity of the local ring ${\mathcal O}_{Z,M}$ given by Samuel. Since ${\mathcal O}_{Z,M}$ can be viewed as the stalk of the structure sheaf ${\mathcal O}_M$ at the generic point of $Z$, the multiplicity of $M$ along $Z$ can also be computed as the multiplicity of $M$ at the generic point of $Z$.

Let $Z_1, \ldots, Z_r$ be pure-dimensional subschemes of $M$ with 
$$
m= \sum_{j=1}^r \dim(Z_j)-(r-1)n \geq 0.
$$
If $Z$ is an $m$-dimensional smooth irreducible component of the intersection $\cap_{j=1}^r Z_j$, let $\pi: \tilde{M} \to M$ be the blow-up of $M$ along $Z$. Denote by $j:E \to \tilde{M}$ the embedding of the exceptional divisor $E$ into $\tilde{M}$. Moreover, let $\eta: E \to Z$ be the projective bundle morphism of $E$ onto $Z$, and
$\tilde{Z}_j$ the blow-up of $Z_j$ along the intersection $Z_j \cap Z$. Then Example 12.4.4. in \cite{Ful} states that
$$
i(Z, Z_1 \cdot \ldots \cdot Z_r; M) = \prod_{j=1}^r e_Z(Z_j) + q,
$$
where $i(Z, Z_1 \cdot \ldots \cdot Z_r; M)$ is the intersection multiplicity as defined in~\cite{Ful}*{Ex.~8.2.1} and $q$ is the coefficient of $[Z]$ in the cycle $\eta_*(j^*(\tilde{Z}_1 \cdot \ldots \cdot \tilde{Z}_r)) \in A_*(Z).
$

\vspace{5mm}

We are now going to prove a technical result
about the blow-up $X_1^n$ of $(\pp^1)^n$ at one
point $q_1$. In what 
follows we will denote simply by $L$
the $1$-dimensional fiber with class
$H_{\{2,\dots,n\}}-E_{1,1}$.
\begin{lem}
\label{lem:bs}
If we denote by $W_s$ the class $H_2+\dots+H_n-sE_1$, then for any 
$1\leq s\leq n-1$ the following hold.

\begin{itemize}
    \item[(i)] The base locus of the linear system $|W_s|$ 
    is the union of 
    fibers with classes $H_I-E_{1,s}$, for $I \subseteq \{2,\dots,n\}$ and $|I|=n-s$.
    \item[(ii)] The general divisor in $|W_s|$ has 
    multiplicity $s$ along the curve $L$.
    \item[(iii)] If $s\leq n-2$, the
    intersection of a general divisor in $|W_s|$ with the 
    base locus of $|W_{s+1}|$ is the 
    base locus of $|W_s|$.
\end{itemize}    
\end{lem}
\begin{proof}
First of all observe that the linear system
$|H_2+\dots+H_n|$ corresponds to 
$\mathcal O(0,1,\dots,1)$ on $(\pp^1)^n$.
A basis for this linear system 
is given by the $2^{n-1}$ monomials
\[
\left\{
\prod_{i\in I, j\in I^c}
x_jy_i,|\,
I\subseteq \{2,\dots,n\}
\right\}
\]
where we set $I^c := \{2,\dots,n\}\setminus I$. 
For any integer $1\leq s\leq n-1$,
the subsystem $|W_s| = |H_2+\dots +H_n-sE_1|$
corresponds to $\mathcal O(0,1,\dots,1)\otimes \mathcal I_{q_1}^s$ (i.e. the sections 
vanishing with multiplicity
at least $s$ at $q_1$). We can suppose that the point that we blow up is
$q_1 := ((1:0),\dots,(1:0)) \in (\pp^1)^n$, so that
the system $|W_s|$ 
has the following monomial basis 
\begin{equation}
\label{eq:mon}
\left\{
\prod_{i\in I, j\in I^c}
x_jy_i,|\,
I\subseteq \{2,\dots,n\},\,
|I| \geq s
\right\}.
\end{equation}
Let us prove $(i)$. Observe that
the monomial $y_2\,\cdots\, y_n$ appears in a
basis for $|W_s|$, for any $1\leq s\leq n-1$.
In order for this monomial to vanish,
it must be $y_j=0$ for some 
$j\in\{2,\dots,n\}$. 
We substitute in the monomial $x_j\prod_{i\neq j}y_i$ and since $x_j$ can
not vanish, we get that at least another 
$y_i$ vanishes. Iterating the above reasoning, 
in order to make all the monomials vanish,
at least $n-s$ variables $y_l$ must vanish,
which gives the statement.

Let us prove $(ii)$. Since the ideal $\mathcal I_L$ of the
curve $L$ is $(y_2,\dots,y_n)$, it is enough
to observe that 
every monomial in~\eqref{eq:mon} belongs to 
the power $\mathcal I_L^s$, and there
is at least one monomial that does not
belong to $\mathcal I_L^{s+1}$.

Finally, since by $(i)$ we know the base locus
of $|W_{s+1}|$, in order to prove $(iii)$ 
it suffices to show that a general element 
in $|W_s|$ 
intersects a fiber $H_I-E_{1,s+1}$, with 
$|I|=n-s-1$, along the fibers 
$H_{I\cup\{j\}}-E_{1,s}$, for $j\notin I$.
Indeed the fiber $H_I-E_{1,s+1}$
can be described by the vanishing of all the 
variables $y_i$, for $i\in I$. If we substitute in a defining
equation for $W_s$, all the monomials but 
$\prod_{i\in I,\,j\in I^c}x_iy_j$ vanish.
Since $x_i\neq 0$ for every $i\in I$, the vanishing
of this monomial gives the statement.
\end{proof}


Let us consider the blow-up $\sigma\colon
\tilde X \to X_1^n$ along $L$, 
with exceptional divisor $E$. 
Since the normal bundle of $L$ in $X_1^n$
is $\mathcal O(-1)^{n-1}$, $E$ is isomorphic to $\pp^1\times\pp^{n-2}$. By abuse of notation, in what follows 
we will denote simply by $H_i$ 
the pull-back $\sigma^*H_i$ and by 
$E_1$ the pull-back $\sigma^* E_1$.

\begin{lem}
\label{lem:tor}
    For any $1\leq s\leq n-1$, 
    let $\tilde W_s$ be the strict transform of $W_s$. The restriction
    of $|\tilde W_s|$ to $E$ cuts on a general
    $\pp^{n-2}$ the linear system of hypersurfaces
    of degree $s$, with $n-1$ general points
    of multiplicity $s-1$.
\end{lem}
\begin{proof}
Since we are blowing up $(\pp^1)^n$ at the
invariant point $q_1$, and then along the
invariant curve $L$, $\tilde X$ is a toric variety. 
The fan $\Sigma\subseteq \qq^n$ of 
$(\pp^1)^n$ has $2n$ rays, generated by 
$\pm e_i$, for $1\leq i \leq n$.
We can suppose that the point $q_1$ 
corresponds to the $n$-dimensional cone 
$\langle -e_i \mid\, i \in \{1,\dots,n\}\rangle$, while $L$ corresponds
to the $(n-1)$-dimensional cone 
$\langle -e_i \mid\, i \in \{2,\dots,n\}\rangle$.
Therefore the fan $\tilde\Sigma$ of the blow up $\tilde X$ 
is obtained from $\Sigma$ by a stellar subdivision of the above two cones with respect to the rays 
$\rho_1:= -\sum_{i=1}^ne_i$ and $\rho_2
:= -\sum_{i=2}^ne_i$ respectively.
Each ray corresponds to 
an invariant divisor, and to a variable 
in the Cox ring $\mathbb C[T_1,S_1,\dots,T_n,S_n,U_1,U_2]$ of $\tilde X$,
as follows:
\[
\begin{array}{rclcll}
e_i & \Rightarrow & H_i & \Rightarrow & T_i, & 
i = 1,\dots,n\\
-e_1 & \Rightarrow & H_1  - E_1& \Rightarrow & S_1\\
-e_i & \Rightarrow & H_i - E_1 - E & 
\Rightarrow & S_i, & i = 2,\dots,n\\
\rho_1 & \Rightarrow & E_1 - E & \Rightarrow 
& U_1 &
\\
\rho_2 & \Rightarrow & E & \Rightarrow 
& U_2. &
\end{array}
\]
We can now describe the sections of 
$|\tilde W_s| = |\sum_{i=2}^n H_i - sE_1 -sE|$ in Cox coordinates as
\begin{equation}
\label{eq:cox}    
\left\{\prod_{i \in I, j\in I^c}T_jS_i (U_1
U_2)^{|I|-s}\mid\,
I\subseteq \{2,\dots,n\},\, |I| \geq s\right\}.
\end{equation}
Let us consider the 
restriction of these sections to the exceptional
divisor $E = \pp^1\times\pp^{n-2}$. 
The fan of $E$ lies inside $\mathbb Q^{n-1}$
and can be obtained taking the
images of the cones containing $\rho_2$, via 
the quotient by $\rho_2$. The quotient
map can be described via the matrix
\[
\begin{pmatrix}
1 & 0 & 0  & \dots & 0 & 0\\
0 & -1 & 0  & \dots & 0 & 1\\
0 & 0 & -1  & \dots & 0 & 1\\
\vdots & \vdots & \vdots & 
\ddots & \vdots & 1\\
0 & 0 & 0 & \dots & -1 & 1
\end{pmatrix}.
\]
The rays of the cones containing $\rho_2$ are $e_1,\rho_1,-e_2,\dots,-e_n$, so that their  
images via the quotient are
\[
\varepsilon_1,\, -\varepsilon_1,\,
\varepsilon_2,\ \dots,\,
\varepsilon_{n-1},\, -\sum_{i=2}^{n-1}\varepsilon_i,
\]
where $\varepsilon_1,\,\dots,\,
\varepsilon_{n-1}$ is the canonical basis of
$\mathbb Q^{n-1}$.
Moreover, if we substitute
$U_2 = 0$ in~\eqref{eq:cox}, we obtain
the restriction of the sections of 
$|\tilde W_s|$ to $E$, namely
\[
\left\{
\prod_{i \in I, j\in I^c}T_jS_i\mid\,
I\subseteq \{2,\dots,n\},\, |I| = s
\right\},
\]
corresponding to $\{\prod_{i\in I}z_i\mid\,|I|=s\}$
on $\pp^{n-2}$, with variables $z_2,\dots,z_n$.
These squarefree monomials are 
a basis of the linear system
of hypersurfaces of degree $s$, having multiplicity
at least $s-1$ at the $n-1$ fundamental points of $\pp^{n-2}$.

\end{proof}

\begin{lem}
 \label{lem:num}
 Let $v = (a_1,\dots,a_N,-b_1,\dots,-b_r)
 \in \zz^{N+r}$ be a vector 
 satisfying the following
 inequalities:
 \begin{enumerate}
  \item $a_i\geq 0$ and 
  $b_j \geq 0$, for any 
  $1\leq i \leq N$ and $1\leq j\leq r$;
  \item $\sum _{i=1}^N a_i \geq 
  \sum_{j=1}^r b_j$. 
 \end{enumerate}
 Then $v$ is a non-negative 
 linear combination of $e_j$ and 
 $e_i-e_j$, for $1\leq i\leq N$
 and $N+1\leq j\leq N+r$.
\end{lem}
\begin{proof}
We argue by induction on $b:=\sum_{j=1}^r b_j$.
If $b = 0$ we have that 
\[
 v = (a_1,\dots,a_N,0\dots,0) 
 = \sum_{i=0}^N a_i(e_i-e_{N+1}) 
+ (\sum_{i=1}^N a_i) e_{N+1}
\]
and since $a_i \geq 0$, the claim follows. 
Let us suppose the statement to be true for
$b > 0$ and let us consider 
$v := (a_1,\dots,a_N,-b_1\dots,-b_r)
\in \zz^{N+r}$ such that
$\sum_{j=1}^r b_j = b+1$. There exist
at least two indexes 
$1\leq i\leq N$ and $1\leq j\leq r$ such that
$a_i,b_j > 0$. We can write 
$v = (a_1,\dots,a_i-1,\dots,a_N,-b_1,\dots,
-(b_j-1),\dots,-b_r) + 
(e_i-e_{N+j})$.
We conclude by the induction hypothesis.

\bigskip


\end{proof}

\section{Fiber-generated effective cones}

In this section we first give the definition of fiber
generated pseudoeffective cone and, after that, we 
present the main results of our work. 

\begin{defin}
\label{def:fib}
 For any $1\leq k \leq n-1$ we define the
 {\em cone of fibers} $\CF_k(X_{\Gamma}^n)
 \subseteq \bEff_k(X_{\Gamma}^n)$ to be
 the cone generated by the classes 
 \[
 \begin{array}{ll}
  H_I-E_{i,k}, & |I| = n-k,\, 1\leq i\leq r\\
  E_{i,k}, & 1\leq i\leq r
 \end{array}
 \]
 i.e. the $k$-dimensional fibers through
 the points and the classes of $k$-dimensional
 linear spaces in the exceptional divisors.
 If the equality $\bEff_k(X^n_{\Gamma}) = 
 \CF_k(X_{\Gamma}^n)$ holds, we say that 
 the pseudoeffective cone $\bEff_k(X^n_{\Gamma})$
 is {\em fiber-generated}.
\end{defin}
Below, you find two different characterizations of fiber-generation: the first is geometric and depends on the nefness of a specific class; the second is numeric and depends on the nonnegativity of a combination of suitable integers.

\begin{lem}
\label{lem:fg}
Let $\Gamma = \{q_1,\dots,q_r\}$ be a set of distinct points
of $(\pp^1)^n$. For any $1\leq k\leq n-1$, the 
pseudoeffective cone $\bEff_k(X^n_{\Gamma})$ is fiber-generated
if and only if the class
\[
 \sum_{|I|=k}H_I - \sum_{j=1}^rE_{j,n-k}
\]
is nef.
\end{lem}
\begin{proof}
First of all observe that
$\bEff_k(X_{\Gamma}^n)$ is fiber-generated
if and only if the nef cone $\Nef_{n-k}(X_{\Gamma}^n)$ coincides with the dual cone of
$\CF_k(X_{\Gamma}^n)$.
Using the matrix of the intersection
product (see Remark~\ref{rem:int}), we can see that the dual of $\CF_k(X_{\Gamma}^n)$ 
has the following $\binom{n}{k}+2^r-1$ 
extremal rays
\[
 H_I \quad \text{for } |I| = k,
\quad
\text{ and }
\quad
\sum_{|I|=k}H_I - \sum_{j\in S}E_{j,n-k},
\quad
 \text{for } S\subseteq\{1,\dots,r\},
\]
so that $\bEff_k(X_{\Gamma}^n)$
is fiber-generated if and only if all the
above classes are nef. Observe that 
any class $H_I$ is nef since it 
is the intersection of the nef divisors 
$H_i$, for $i\in I$ (see for instance~\cite{FL}). 
We now claim that any class of the form
$\sum_{|I|=k}H_I - \sum_{j\in S}E_{j,n-k}$,
with $S\subseteq\{1,\dots,r\}$ is nef if and only if 
the class $\sum_{|I|=k}H_I - \sum_{j=1}^rE_{j,n-k}$ 
is.
Indeed, given any $k$-dimensional subvariety
$V\subseteq X_{\Gamma}^n$, not contained in an
exceptional divisor, we have $[V] = \sum_{|J|=n-k}a_JH_J-
\sum_{j=1}^r b_jE_{j,k} \in\bEff_{k}(X_{\Gamma}^n)$, 
where $a_J,b_j\geq 0$. For any 
$S\subseteq\{1,\dots,r\}$ we can write
\[
[V]\cdot\left(\sum_{|I|=k}H_I - \sum_{j\in S}E_{j,n-k}\right)
= 
[V]\cdot \left(\sum_{|I|=k}H_I - \sum_{j=1}^rE_{j,n-k}\right)
+\sum_{j\notin S}b_j.
\]
We conclude observing that if 
$\sum_{|I|=k}H_I - \sum_{j=1}^rE_{j,n-k}$ is nef
then the above intersection products are all non-negative.
\end{proof}

\begin{lem}
\label{prop:ineq}
The pseudoeffective cone $\bEff_k(X_{\Gamma}^n)$
is fiber-generated if and only if for any 
$k$-dimensional variety $Y$, not contained in 
an exceptional divisor and having class
$[Y] = \sum_{|I|=n-k}a_IH_I-\sum_{j=1}^rb_jE_{j,k}$,
the inequality 
$\sum_{|I|=n-k}a_I \geq \sum_{j=1}^rb_j$
holds.
\end{lem}
\begin{proof}
Observe that if $Y$ is not contained in 
an exceptional divisor, then 
$[Y] = \sum_{|I|=n-k}a_IH_I-\sum_{j=1}^rb_jE_{j,k}$,
with $a_I,b_j\geq 0$. If the inequality
$\sum_{|I|=n-k}a_I \geq \sum_{j=1}^rb_j$ holds,
then we can apply Lemma~\ref{lem:num} 
with $N = \binom{n}{k}$. We conclude that 
$[Y]$ is a non-negative linear combination
of $E_{j,k}$ and $H_I-E_{j,k}$, for $|I| = n-k$
and $1\leq j\leq r$, so that $\bEff_k(X_{\Gamma}^n)$
is fiber-generated.

On the other hand, if $\bEff_k(X_{\Gamma}^n)$
is fiber-generated, then by Lemma~\ref{lem:fg}
the class 
$\sum_{|J| = k}H_J-\sum_{j=1}^rE_{j,n-k}$
is nef. Taking the intersection product of
the latter with $[Y]$ we get the inequality
in the statement.
\end{proof}

Our first result is that if we blow 
up few very general points, then all the
pseudoeffective cones are fiber-generated.
\begin{prop}
\label{prop:tor}
 If $r\leq 2$, the cone 
 $\bEff_k(X^n_r)$ is fiber-generated for any 
 $n \geq 2$ and $1\leq k\leq n-1$.
 \end{prop}
\begin{proof}
If $r\leq 2$, the blow up $X^n_r$ is a toric variety.
In order to prove the statement it is enough to
deal with the case $r=2$.
By~\cite{Li}*{Prop.~3.1},
for any $1\leq k\leq n-1$, the cone 
$\bEff_k(X^n_2)$ is generated
by the classes of invariant $k$-cycles, i.e. 
the classes of cycles corresponding to the
cones of codimension $k$ in the fan of 
$X^n_2$. 

We recall that the fan $\Sigma\subseteq \qq^n$ of 
$(\pp^1)^n$ has $2n$ rays, generated by 
$\pm e_i$, for $1\leq i \leq n$.
For any $0\leq k\leq n-1$ there are 
$2^{n-k}\cdot \binom{n}k$
cones of codimension $k$ in $\Sigma$, namely
\[
 \langle \pm e_i \mid i \in I \rangle,
 \quad
|I| = n-k.
\]
If we blow up two points $q_1,\, q_2\in (\pp^1)^n$,
we can suppose that they correspond to the $n$-dimensional cones $C_1 = \langle e_i \mid\, i \in \{1,\dots,n\}\rangle$ and $C_2 = \langle -e_i \mid\, i \in \{1,\dots,n\}\rangle$ respectively. 
Therefore the fan $\Sigma'$ of the blow up $X_2^n$ 
is obtained from $\Sigma$ by a stellar subdivision of $C_1$ and $C_2$ with respect to the rays 
$\rho:=\sum_i e_i$ and $-\rho$ respectively, so that it has $2n+2$ rays.
For any $k>1$ the cones of codimension $k$
in $\Sigma'$ can be described as the cones 
of codimension $k$ of $\Sigma$, together with 
the cones of the form
$\langle e_i,\rho \mid i \in
I\rangle$ and 
$\langle -e_i,-\rho \mid i 
\in I\rangle$, for all the
subsets $I\subseteq
\{1,\dots,n\}$ such that $|I|
= n-k-1$. We summarize these cones
and the corresponding cycles in the following table.
\[
\begin{array}{rcll}
     \langle e_i\, \mid i \in I \rangle & 
     \Leftrightarrow & H_I - E_{1,k}, &
     |I| = n-k,\\
     \langle -e_i\, \mid i \in I \rangle & 
     \Leftrightarrow & H_I - E_{2,k}, &
     |I| = n-k,\\
     \langle e_i,\rho\, \mid i \in I\rangle &
     \Leftrightarrow & E_{1,k}, & 
     |I| = n-k-1,
     \\
     \langle -e_i,-\rho\,\mid i \in I\rangle &
     \Leftrightarrow & E_{2,k},& 
     |I| = n-k-1,\\
     \langle e_i,-e_j\, \mid i\in I,\, j\in J\rangle &
     \Leftrightarrow & H_{I\cup J}, &
     |I|+|J| = n-k,\, I \cap J = \emptyset.\\

\\
 \end{array}
\]
We conclude by observing that the classes above
are exactly the generators of 
the cone of fibers $\CF_k(X_2^n)$.
\end{proof}

\begin{rem}
If $\Gamma =\{q_1,q_2\}$ consists 
of $2$ points lying on a $s$-dimensional 
fiber $F$ of a projection 
$p_J\colon (\mathbb P)^n
\to (\mathbb P)^{n-s}$, with $|J| = n-s$, the
variety $X^n_{\Gamma}$ is still toric
and we can reason in the same way.
In particular $\bEff_k(X^n_{\Gamma})$ has the
the same generators as before, for $k < s$,
while for $k\geq s$ the cone $\bEff_k(X^n_{\Gamma})$
is generated by the following classes:
\[
\begin{array}{ll}
H_I - E_{1,k},\ H_I - E_{2,k}, &
     |I| = n-k,\  I\not \subset J\\
H_I - E_{1,k} - E_{2,k}, &
     |I| = n-k,\ I \subset J\\
E_{1,k},\ E_{2,k}. & \\
\\
 \end{array}
\]

\end{rem}

We immediately get the following characterization in the case of cycles 
of codimension $1$.
\begin{cor}
\label{cor:div}
For any $n\geq 2$ the cone
$\bEff_{n-1}(X^n_r)$ is fiber-generated if and only if $r\leq 2$.
\end{cor}
\begin{proof}
 By Proposition~\ref{prop:tor} we only need
 to show that if $r\geq 3$, the cone
 $\bEff_{n-1}(X^n_r)$ is not fiber-generated.
 The class 
 $H_1 + H_2 - E_1 - E_2 - E_3$ corresponds to
 a hyperplane of $\pp^n$ passing through $4$
 points via the map $\varphi$ described in~\eqref{eq:1}, 
 and in particular it is effective.
 We conclude observing that 
 $H_1 + H_2 - E_1 - E_2 - E_3$ does not 
 belong to the cone of fibers 
 $\CF_{n-1}(X_2^n)$.
\end{proof}

In the case of curves, it is possible to 
give a sharp bound for fiber generation, 
by means of the following.

\begin{theo}
 \label{thm:linear}
 For any $n\geq 2$, the 
 cone $\overline{\Eff}_1(X^n_r)$ is 
 fiber-generated
 if and only if $r\leq n!$.
\end{theo}
\begin{proof}
 We argue as in the proof 
 of~\cite{CLO}*{Prop.~4.1}. 
 Let us consider
 the embedding $(\pp^1)^n\subseteq
 \pp^{2^n-1}$ and let us take $n$ general
 hyperplane sections $W_1,\dots,W_n$. 
 Since the $W_i$ are general, 
 $\Gamma:= W_1\cap \dots \cap W_n$
 consists of $n!$ distinct points of $(\pp^1)^n$.
 We claim that the class 
 $D : = H_1+\dots+H_n-E_1-\dots -E_{n!}$ is nef
 on $X_{\Gamma}^n$.
 Indeed, given any curve $C$ on $X_{\Gamma}^n$, 
 not contained in an exceptional divisor, there
 exists at least one $W_i$ whose proper
 transform does not contain $C$. Since the class of $W_i$ is $D$, we get $D\cdot C \geq 0$
 and the claim follows. Arguing by semicontinuity
 we deduce that the class
 $H_1+\dots+H_n-E_1-\dots -E_{n!}$ is nef on
 $X^n_{n!}$ too.
 
 Therefore, if $r\leq n!$, the class
 $H_1+\dots+H_n-E_1-\dots -E_r$ is nef 
 on $X_r^n$, so that, by Lemma~\ref{lem:fg},
 $\bEff(X_r^n)$ is fiber-generated.
 If otherwise $r > n!$, the above class is not 
 nef since its $n$-th power is negative,
 and we conclude again by means of 
 Lemma~\ref{lem:fg}.
\end{proof}
Going back to any $1\leq k\leq n-1$, in order to prove 
our general bound (Theorem~\ref{thm:bound}) we need 
the following results
about the blow up $X_1^n$ of $(\pp^1)^n$
at one point $q_1$. We recall that 
$L$ is the curve with class  
$H_{\{2,\dots,n\}} - E_{1,1}$ (i.e. the 
strict transform of the fiber of the 
projection $p_{2,\dots,n}\colon (\pp^1)^n
\to (\pp^1)^{n-1}$, passing through $q_1$).

\begin{lem}
\label{lem:con}
 Let $Y$ be a purely $k$-dimensional 
 subvariety on the blow up $X_1^n$
 of $(\pp^1)^n$ at a point $q_1$, 
 with class
 $[Y] = \sum_{|I| = n-k}a_IH_I -
 b_1E_{1,k}$. 
 If 
 $b_1 - \sum_{1\in I}a_I  = \beta > 0$, 
 then the curve $L$
 is contained in $Y$ with multiplicity
 at least $\beta$. 
\end{lem}
\begin{proof}
We recall that for any $1\leq s\leq n-1$, we 
set $W_s := \sum_{i=2}^nH_i - sE_1$. We have
\[
\begin{array}{rcl}
W_1\cdot W_2\ \cdots\  W_k\cdot Y
& = & ((\sum_{i=2}^nH_i)^k - (k!)E_{1,n-k})\cdot Y\\[3mm]
& = & k!\left(\displaystyle
\sum_{\genfrac{}{}{0pt}{}{I\subseteq\{2,\dots,n\}}{|I| = k}}
H_I-E_{1,n-k}\right)\cdot Y\\[6mm]
& = & k!\left(\displaystyle\sum_{1\in I}a_I - b_1\right)\\[4mm]
& = & -k! \beta < 0.
\end{array}
\]

By Lemma~\ref{lem:bs}, the general variety $V$ in the intersection product $W_1 \ 
\cdots \ W_{k-1}$ has codimension $k-1$, as the divisors $W_s$ are general. 
For the same
reason the intersection $V\cap Y$ has dimension $1$ and it contains $L$ as an
irreducible component. As explained in \cite{Ful}*{p. 118}, the intersection multiplicity $i(L, V\cdot Y; X_1^n)$ is the coefficient of $L$ in the intersection class $V \cdot Y$.

By Lemma~\ref{lem:bs}
$(iii)$, $L$ is the only irreducible component
of $V\cap Y$ contained in the base locus of
$|W_k|$, so that the intersection product
of $W_k$ with every irreducible component
but $L$ is nonnegative. We then have the following
inequality
\[
 -k!\beta
 = W_k\cdot(V\cdot Y)
 \geq 
 i(L, V \cdot Y; X_1^n)(W_k\cdot L),
\]
and since $W_k\cdot L = -k$, we deduce that 
\begin{equation}
    \label{eq:q}
    i(L, V \cdot Y; X_1^n) \geq (k-1)!\beta.
\end{equation}
Let us consider the commutative diagram
\[
\begin{tikzcd}
E \arrow[r, hookrightarrow, "j"] \arrow[d, "\eta"'] & \tilde{X} 
\arrow[d, "\sigma"]\\
L \arrow[r, hookrightarrow] & X_1^n
\end{tikzcd}
\]  
where $\sigma$ is the blowing up of $L$, while $\eta$ is the projective bundle morphism.
By~\cite{Ful}*{Ex.~12.4.4}, we have
\[
i(L,V\cdot Y;X_1^n) = i(L, W_1 \ \cdots \ W_k 
\cdot Y; X_1^n) = e_L(W_1)\ \cdots\ e_L(W_{k-1})\cdot e_L(Y) + q,
\]
where $q$ is the coefficient of $[L]$ in the cycle
$\eta_*(j^*(\tilde V\cdot\tilde Y))$.
We claim that in our case $q=0$, so that, 
since by Lemma~\ref{lem:bs} 
$(ii)$, $e_L(W_s) = s$ for any $s\geq 1$, substituting in~\eqref{eq:q} we get the statement, $e_L(Y)\geq \beta$.

In order to prove the claim, we first 
recall that $E$ is isomorphic to  $\pp^1
\times\pp^{n-2}$. The pull-backs $j^*(\tilde V)$ and $j^*(\tilde Y)$ cut on a general $\pp^{n-2}$
two cycles of codimension $k-1$ and $n-k$ respectively. We conclude by
showing that these two cycles do not intersect,
so that the pushforward $\eta_*(j^*(\tilde V\cdot\tilde Y))$ is $0$-dimensional, and in
particular the coefficient $q$ of $[L]$ must be 
$0$.
Indeed, by Lemma~\ref{lem:tor}, 
$j^*(\tilde W_s)$ cuts on $\pp^{n-2}$
the linear system of hypersurfaces of degree $s$,
with $n-1$ points of multiplicity $s-1$. 
This implies that $j^*(\tilde V)$ does
not intersect the $(k-3)$-dimensional 
base locus of $j^*(\tilde W_{k-1})$ (which is 
the union of 
$(k-3)$-dimensional linear spaces). Hence 
$j^*(\tilde{V})$ and $j^*(\tilde{Y})$ do not intersect 
on a general $\pp^{n-2}$.

\end{proof}
We are now able to give an upper bound 
that works for any $n$ and $k$ (even if
it is not sharp in general).
\begin{theo}
\label{thm:lb}
    The cone $\bEff_{k}(X^n_r)$ is fiber-generated for any $r \leq n-k+1$.
\end{theo}
\begin{proof}
By Theorem~\ref{thm:linear} and 
Proposition~\ref{prop:tor} the statement is
true when $k=1$ or $k=n-1$, so that in what follows
we can assume $2\leq k\leq n-2$.
We proceed by induction on $n$ as in 
\cite{CLO}*{Thm.~4.3}, the base of induction 
being $n=2$.
Assume the theorem is true for $\bEff_k(X_r^m)$, 
for any $k<m<n$. Let $\Gamma\subseteq (\pp^1)^n$ 
be a set 
consisting of $r-1$ very general points $q_1, 
\ldots, q_{r-1}$ in a fiber $F$ of the projection 
morphism $p_{\{1\}}: (\pp^1)^n \to {\mathbb P}^1$, 
and a very general point $q_r$ not contained 
in $F$. Notice that 
$F \simeq (\pp^1)^{n-1}$.
We denote by $F'$ be the proper transform of $F$ in 
$X_{\Gamma}^n$, the blow up of 
$(\pp^1)^n$ along $\Gamma$,
so that the class 
of $F'$ in $X_{\Gamma}^n$
is $H_1-\sum_{j=1}^{r-1}E_j$. 

Let $Y$ be an irreducible $k$-dimensional subvariety of $X_{\Gamma}^n$, not contained in an exceptional divisor, so that
$$
[Y] = 
\sum_{|I|=n-k}a_IH_I - \sum_{j=1}^rb_jE_{j,k}, 
$$
where $a_I, b_j$ are nonnegative integers. By 
Proposition~\ref{prop:ineq}
we need to prove that 
\begin{equation}
\label{eq:ineq}
\sum_{|I|=n-k}a_I \geq \sum_{j=1}^r b_j.
\end{equation}
We distinguish two cases.\\
(i) $Y$ is contained in $F'$. Since the points 
$q_1,\dots,q_{r-1}$ are very general in $F$
we have that $F' \simeq X^{n-1}_{r-1}$. 
By the induction hypothesis, 
$\bEff_k(X^{n-1}_{r-1})$ is fiber-generated,
so that inequality~\eqref{eq:ineq} holds.\\
(ii) $Y$ is not contained in $F'$. In this case
we set $Z := Y \cap F'$, so that $Z$ 
is an effective cycle of dimension $k-1$ in 
$X^{n-1}_{r-1}$. 
The $1$-dimensional fiber 
of the projection $\pi_{\{2,\dots,n\}}: ({\mathbb P^1})^{n} \to ({\mathbb P^1})^{n-1}$, passing through the point $q_r$
intersects $F$ in a unique point $\bar q_r$.  
The blow up of $F'$ at the point 
corresponding to $\bar q_r$ coincides with 
the blow-up $X_{r}^{n-1}$ of $F$ along 
the $r$ very general points $q_1, \ldots, q_{r-1},\bar q_r$. 
By Lemma~\ref{lem:con}, $Z$ has multiplicity
$\bar b_r\geq \beta =\max(0,b_r-\sum_{1\in I}
a_I)$ at $\bar q_r$. Therefore 
the strict transform $Z'$ of $Z$ is a $(k-1)$-dimensional effective cycle in $X^{n-1}_r$,
with class
\[
[Z']
=
\sum_{1\notin I}a_IH_I
-
\sum_{j=1}^{r-1}b_jE_{j,k-1} - \bar b_rE_{r,k-1},
\]
(with a slight abuse of notation we are 
considering the 
classes $H_I$ of $N_{k-1}(X^{n}_r)$, such that 
$1\in I$, as classes in $N_{k-1}(X^{n-1}_r)$).
Notice that $r \leq (n-1)-(k-1)+1$, so the inductive procedure can be applied. Therefore, $\bEff_{k-1}(X^{n-1}_r)$ is fiber-generated and 
the coefficients of $[Z']$ satisfy
the inequality
$\sum_{1\notin I}a_I
\geq
\sum_{j=1}^{r-1}b_j+\bar b_r$.
We can then write
\[
0 \leq 
\sum_{1\notin I}a_I-
\sum_{j=1}^{r-1}b_j - \bar b_r
\leq 
\sum_{1\notin I}a_I-
\sum_{j=1}^{r-1}b_j - (b_r - 
\sum_{1\in I}a_I)
=
\sum a_I - \sum_{j=1}^r b_j,
\]
as required.
Hence we can apply Lemma \ref{lem:num} and deduce that  $\bEff_k(X^n_{\Gamma})$ is fiber-generated. Reasoning 
by semicontinuity as in~\cite{CLO}*{Cor.~2.3}, we conclude that $\bEff_k(X_r^n)$
is fiber-generated too.

\end{proof}

\section{Other results}
In this section we first show that for any $n$ and $k$, if
we blow up a sufficiently big number $r$ of points, the
cone $\bEff_k(X_r^n)$ is not fiber-generated. As a consequence, we 
consider the particular case of $\bEff_2(X_4^4)$, and 
finally we discuss Question~\ref{que:md}.
\begin{prop}
\label{prop:not}
The cone $\bEff_k(X_r^n)$ is not fiber-generated for
$r > (n-k+1)!$.
\end{prop}
\begin{proof}
If $k = 1$ or $k = n-1$, the assertion follows from
Theorem~\ref{thm:linear} and    
Corollary~\ref{cor:div}, so we can restrict to the 
case $2\leq k\leq n-2$. Let us fix the projection
$\pi_J \colon (\pp^1)^n \to (\pp^1)^{n-k+1}$, where
$J = \{1,\dots,n-k+1\}$, and let $X_r^{n-k+1}$ be the
blow up of  $(\pp^1)^{n-k+1}$ at the points 
$q_1',\dots,q_r'$, images of $q_1,\dots,
q_r$ via $\pi_J$. 
Since by hypothesis $r > (n-k+1)!$, by 
Theorem~\ref{thm:linear} the cone $\bEff_1(X_r^{n-k+1})$
is not fiber-generated, so that by 
Proposition~\ref{prop:ineq} there exists a
curve class $[C] = \sum_I a_IH_I - \sum_{j=1}^rb_jE_{j,1}$, where $I\subseteq
J$, and $|I| = n-k$, satisfying the inequality
\[
\sum_I a_I < \sum_{j=1}^rb_j.
\]
Observe that the image $C'$ of $C$ in $(\pp^1)^{n-k+1}$
is a curve passing through $q_j'$ with multiplicity at least $b_j$, for any $1\leq j\leq r$. 
The fiber product $Z:=(\pp^1)^{k-1}\times
C \subseteq (\pp^1)^{n}$ is a $k$-dimensional 
variety having multiplicity at least $b_j$ at each
point $q_j$, so that its class can be written as
$[Z] = \sum_I \alpha_IH_I - \sum_{j=1}^rb_jE_{j,k}$,
with $I\subseteq \{1,\dots,n\}$ and $|I| = n-k$. Since 
$Z$ is the product of $(\pp^1)^{k-1}$ and $C$,
we have that 
\[
\alpha_I = 
\left\{
\begin{array}{ll}
a_I & \text{ if } I\subseteq J,\\
0 & \text{ otherwise}
\end{array}
\right.
\]
and in particular the coefficients satisfy
$\sum_I \alpha_I = \sum_I a_I < \sum_{j=1}^rb_j$. 
Since $[Z]$ belongs to the cone $\bEff_k(X_r^n)$,
by Proposition~\ref{prop:ineq} we conclude that the latter is not fiber-generated.
\end{proof}
As we pointed out in the Introduction,
when $n=4$, Theorem~\ref{thm:bound} 
and the above proposition show that
only three cases remain open, i.e. 
$\bEff_2(X_r^4)$, for $r = 4,5,6$. 
We are now going to deal with the first of 
these cases.
\begin{prop}
\label{prop:4-4}
The cone $\bEff_2(X^4_4)$ is fiber-generated.
\end{prop}
\begin{proof}
Let us consider two divisors 
$D_1$ and $D_2$ having classes 
\begin{align*}
    [D_1] & = 
    H_1+H_2+H_3+H_4-2E_1-2E_2-E_3-E_4,\\
    [D_2] & = 
    H_1+H_2+H_3+H_4-E_1-E_2-2E_3-2E_4,
\end{align*}
i.e. $D_1$ (resp. $D_2$) can be described
as the intersection of $X$ with
a hyperplane passing through $q_3$ and $q_4$ 
(resp. $q_1$ and $q_2$) and tangent
to $X$ at $q_1$ and $q_2$
(resp. at $q_3$ and $q_4$).
We have that $\dim |D_1| = 3$ and its base 
locus is $1$-dimensional. Indeed, with the 
help of computer algebra system MAGMA~\cite{mag}
it is possible to see that it is the  
union of the $8$ fibers through 
$q_1$ and $q_2$ and $2$ 
rational normal curves of degree
$4$. Since none
of the components of the base locus is 
contained in $D_2$, we have that $D_1|_{D_2}$
has only finitely many base points
and in particular it is nef.

Let us consider now a surface $Y$ with class
$\sum_{|I|=2} a_IH_I - \sum b_kE_{k,2}$. If $Y$ is not contained in $D_2$,
its restriction to $D_2$ is an
effective $1$-cycle, so that its 
intersection product with $D_1|_{D_2}$ is non-negative.
Therefore we can write
\[
 0 \leq 
Y|_{D_2}\cdot D_1|_{D_2}
=
2\left(\sum_{|I|=2}
a_I - \sum_{k=1}^4b_k\right)
\]
and by Lemma~\ref{lem:num} we
deduce that $[Y]$ lies in the cone 
spanned by fiber classes.

If $Y$ is not contained in $D_1$
we can conclude in the same way,
so that we only have to consider
the case in which $Y\subseteq
D_1\cap D_2$. Since the latter
intersection is irreducible, we
have $Y = D_1\cap D_2$, so that 
$a_I = 2$ for any $I$ such that 
$|I|=2$, and $b_k = 2$ for any
$1\leq k\leq 4$. Therefore
$\sum a_I - \sum b_k = 4 > 0$
and again we can apply 
Lemma~\ref{lem:num}.
\end{proof}
Let us go back now to Question~\ref{que:md}.
As proved in~\cite{CT}, $X_r^n$ is a
Mori dream space exactly in the following cases:
\begin{center}
\begin{longtable}{|c|c|c|c|c|}
\hline
$n$ & $2$ & $3$ & $4$  & $\geq 5$ \\
\hline
$r \leq $ & $7$ & $6$ & $5$ &  $4$\\
\hline
\end{longtable}
\end{center}
Clearly, if $n = 2$ there is nothing to prove, while
if $n=3$ an affirmative answer to Question~\ref{que:md} 
follows from Theorem~\ref{thm:linear}. 
If $n = 4$, by the same theorem and 
Proposition~\ref{prop:4-4}, the only open case
is $k=2,\, r = 5$. Finally, if $n\geq 5$ and $r\leq 4$,
the bound of Theorem~\ref{thm:lb} implies that 
for any $k\leq n-3$ the pseudoeffective cone
$\bEff_k(X_r^n)$ is fiber-generated, so that in 
particular it is polyhedral. Therefore
the only open case is $k=n-2$ and $r=4$.
We remark that, as proved in Lemma \ref{lem:fg}, $\bEff_k(X_r^n)$ is fiber-generated if and only if the numerical class in Lemma \ref{lem:fg} is nef. Therefore, it is natural to ask the following.
\begin{quest}
Is the numerical class 
$\sum_{|I| = n-2}H_I - \sum_{j = 1}^rE_{j,2}$
nef when $n=4,\, r= 5$ or $n\geq 5,\, r = 4$?
\end{quest}
An affirmative answer to this question would allow us 
to answer Question~\ref{que:md} too (see the Introduction). We hope to return to this problem in the next future.


\begin{bibdiv}
\begin{biblist}


\bib{mag}{article}{
   author={Bosma, W.},
   author={Cannon, J.},
   author={Playoust, C.},
   title={The Magma algebra system. I. The user language},
   note={Computational algebra and number theory (London, 1993)},
   journal={J. Symbolic Comput.},
   volume={24},
   date={1997},
   number={3-4},
   pages={235--265},
   issn={0747-7171},
   review={\MR{1484478}},
   doi={10.1006/jsco.1996.0125},
}


\bib{CT}{article}{
   author={Castravet, A.-M.},
   author={Tevelev, J.},
   title={Hilbert's 14th problem and Cox rings},
   journal={Compos. Math.},
   volume={142},
   date={2006},
   number={6},
   pages={1479--1498},
   issn={0010-437X},
   review={\MR{2278756}},
   doi={10.1112/S0010437X06002284},
}



\bib{CLO}{article}{
   author={Coskun, I.},
   author={Lesieutre, J.},
   author={Ottem, J. C.},
   title={Effective cones of cycles on blowups of projective space},
   journal={Algebra Number Theory},
   volume={10},
   date={2016},
   number={9},
   pages={1983--2014},
   issn={1937-0652},
   review={\MR{3576118}},
   doi={10.2140/ant.2016.10.1983},
}


\bib{DELV}{article}{
   author={Debarre, O.},
   author={Ein, L.},
   author={Lazarsfeld, R.},
   author={Voisin, C.},
   title={Pseudoeffective and nef classes on abelian varieties},
   journal={Compos. Math.},
   volume={147},
   date={2011},
   number={6},
   pages={1793--1818},
   issn={0010-437X},
   review={\MR{2862063}},
   doi={10.1112/S0010437X11005227},
}

\bib{FL}{article}{
   author={Fulger, M.},
   author={Lehmann, B.},
   title={Positive cones of dual cycle classes},
   journal={Algebr. Geom.},
   volume={4},
   date={2017},
   number={1},
   pages={1--28},
   issn={2313-1691},
   review={\MR{3592463}},
   doi={10.14231/AG-2017-001},
}

\bib{Ful}{book}{
   author={Fulton, William},
   title={Intersection theory},
   series={Ergebnisse der Mathematik und ihrer Grenzgebiete. 3. Folge. A
   Series of Modern Surveys in Mathematics [Results in Mathematics and
   Related Areas. 3rd Series. A Series of Modern Surveys in Mathematics]},
   volume={2},
   edition={2},
   publisher={Springer-Verlag, Berlin},
   date={1998},
   pages={xiv+470},
   isbn={3-540-62046-X},
   isbn={0-387-98549-2},
   review={\MR{1644323}},
   doi={10.1007/978-1-4612-1700-8},
}

\bib{HK}{article}{
   author={Hu, Y.},
   author={Keel, S.},
   title={Mori dream spaces and GIT},
   note={Dedicated to William Fulton on the occasion of his 60th birthday},
   journal={Michigan Math. J.},
   volume={48},
   date={2000},
   pages={331--348},
   issn={0026-2285},
   review={\MR{1786494}},
   doi={10.1307/mmj/1030132722},
}

\bib{LM}{article}{
   author={Laface, A.},
   author={Moraga, J.},
   title={Linear systems on the blow-up of 
   $(\mathbb {P}^1)^n$},
   journal={Linear Algebra Appl.},
   volume={492},
   date={2016},
   pages={52--67},
   issn={0024-3795},
   review={\MR{3440147}},
   doi={10.1016/j.laa.2015.11.009},
}


\bib{Li}{article}{
   author={Li, Q.},
   title={Pseudo-effective and nef cones on spherical varieties},
   journal={Math. Z.},
   volume={280},
   date={2015},
   number={3-4},
   pages={945--979},
   issn={0025-5874},
   review={\MR{3369360}},
   doi={10.1007/s00209-015-1457-0},
}


\bib{O}{article}{
   author={Ottem, J. C. },
   title={Ample subvarieties and $q$-ample divisors},
   journal={Adv. Math.},
   volume={229},
   date={2012},
   number={5},
   pages={2868--2887},
   issn={0001-8708},
   review={\MR{2889149}},
   doi={10.1016/j.aim.2012.02.001},
}

\bib{P}{article}{
   author={Pintye, N.},
   author={Prendergast-Smith, A.},
   title={Effective cycles on some linear blowups of projective spaces},
   journal={Nagoya Math. J.},
   volume={243},
   date={2021},
   pages={243--262},
   issn={0027-7630},
   review={\MR{4298661}},
   doi={10.1017/nmj.2019.41},
}


\end{biblist}
\end{bibdiv}
\end{document}